\documentclass[conference]{IEEEtran}

\IEEEoverridecommandlockouts                              
\usepackage{graphics} 
\usepackage{epsfig} 
\usepackage{amsmath} 
\usepackage{amssymb}  
\usepackage[latin1]{inputenc}
\usepackage{amsfonts}
\usepackage{subfigure}
\usepackage{color}
\usepackage{extarrows}

\newtheorem{thh}{Theorem}
\newtheorem{prop}{Proposition}

\title{\LARGE \bf
Identification of fractional order  systems  using modulating functions method}

\author{Da-Yan Liu, Taous-Meriem~Laleg-Kirati, Olivier Gibaru and Wilfrid
Perruquetti
\thanks{D.Y. Liu, and T.M. Laleg-Kirati are with Computer, Electrical and Mathematical Sciences and Engineering Division,  King Abdullah university of science and technology, KSA
 {\tt\small Dayan.Liu@kaust.edu.sa};
 {\tt\small taousmeriem.laleg@kaust.edu.sa}}
\thanks{O. Gibaru is with LSIS  (CNRS, UMR 7296), Arts et Métiers
ParisTech, 8 Boulevard Louis XIV, 59046 Lille Cedex,
France
        {\tt\small olivier.gibaru@ensam.eu}}%
\thanks{W. Perruquetti is with LAGIS (CNRS, UMR 8146), \'{E}cole Centrale de Lille, BP 48,
Cit\'e Scientifique, 59650 Villeneuve d'Ascq, France
        {\tt\small wilfrid.perruquetti@inria.fr}}%
\thanks{O. Gibaru and W. Perruquetti are with L'\'{E}quipe Projet Non-A, INRIA Lille-Nord Europe,
40, Avenue Halley,  59650 Villeneuve d'Ascq, France. }
}

\begin{document}

\maketitle
\thispagestyle{empty}
\pagestyle{empty}

\begin{abstract}
The modulating functions method has been used for the identification  of linear and nonlinear systems. In this paper, we generalize this method to the on-line identification of fractional order systems based on the Riemann-Liouville fractional
derivatives. First, a new fractional integration by parts formula involving the fractional derivative of a modulating function is given. Then,  we apply this formula to a fractional order system, for which the fractional derivatives of the input and the output can be transferred into the ones of  the modulating functions. By choosing a set of  modulating functions, a linear system of algebraic equations is obtained. Hence, the unknown  parameters  of a fractional order system can be estimated by solving a linear system. Using this method, we do not need any  initial values which are usually unknown and not equal to zero. Also we do not need to estimate the fractional derivatives of noisy output. Moreover, it is shown that the proposed estimators are robust against high frequency sinusoidal noises and the ones due to a class of stochastic processes. Finally, the efficiency and the stability of the proposed method is confirmed by some numerical simulations.
\end{abstract}

\section{Introduction} \label{section1}

Fractional differential equations and fractional  integrals  are gaining importance in research community because of their capacity to accurately describe  real world processes. The flow of fluid in a porous media, the conduction of heat in a semi-infinite slab, the voltage-current relation in a semi infinite transmission line are such examples of processes naturally modeled by fractional differential equations or fractional integrals.

This paper is dealing with the identification of fractional order dynamical systems.   The identification of such systems has been used for instance, for the estimation of the state of charge of lead acid batteries \cite{SaAoOuGrRaRo:06}, and for the identification in  thermal systems \cite{GaPo:11,GaPoKa:11}.  The goal of system  identification is to estimate the parameters of a model from system input/output measurements. Different methods have been proposed for the identification of fractional order systems. Most of them consist in the generalization to fractional order systems of standard methods that were used in the identification of systems with integer order derivatives.  We can classify these methods into time domain methods and frequency domain methods. Time-domain methods have been introduced for example in \cite{OuLeMa:96, DjVoCh:12},  where a method based on the discretization of a fractional differential equation using Grunwald definition has been introduced, and the parameters have been estimated using least square approach. In \cite{TrPoOuLe:99}, a method based on the approximation of a fractional integrator by a rational model has been proposed.  In \cite{AoMaLeOu:07},  the use of methods based on  fractional orthogonal bases has been introduced.  Other techniques can be also found for example in   \cite{CoOuPoBa:01}, \cite{CoOuBaBa:00},  \cite{MaAoSaOu:06} and the references therein.


In this paper, we are interested in the identification of fractional order systems using modulating functions method in case of noisy measurements. Modulating functions method has been developed by Shinbrot  in \cite{Shinbrot54, Shinbrot57} to estimate the parameters of a state space representation. Thanks to the properties of  modulating functions,
the fractional differential equation defining a fractional order system is transformed into a linear system of algebraic equations.
Hence, instead of solving a fractional differential equation where the initial values are often unknown, the problem of identification is transformed into solving a linear system  where the initial conditions are not required. Generalization of modulating functions method to fractional systems has been already proposed, for example in \cite{Jan:10}. However, the authors of this paper proposed only to reduce  the orders of the  derivatives in a fractional differential equation. Moreover, the noisy case has not been considered.

In the next section, basic definitions of fractional derivatives and modulating functions are recalled. Then in Section \ref{section3},  modulating functions method is applied  to the identification of fractional order linear systems. Error analysis in both the continuous and discrete cases are presented in Section \ref{section4}. Numerical results are presented in Section \ref{section5}, followed by  conclusions summarizing the main results obtained.

\section{Preliminary } \label{section2}

In this section, first we recall the definitions and some useful properties of the  Riemann-Liouville  fractional derivative and  modulating functions. Then, a new  fractional  integration by parts formula is given.

\subsection{Riemann-Liouville fractional derivative}

Let $f$ be a continuous function defined on $\mathbb{R}$, then
the Riemann-Liouville fractional
derivative  of $f$ is defined by (see \cite{Podlubny2} p. 62): $\forall \, t \in \mathbb{R}_+^*$,
\begin{equation}\label{Eq_RiemannLiouvillle}
{\mathrm{D}}_{t}^{\alpha}f(t):=\frac{1}{\Gamma(l-\alpha)}\frac{d^{l}}{dt^{l}}%
\int_{0}^{t}\left(  t-\tau\right)  ^{l-\alpha-1}%
f(\tau)\,d\tau,
\end{equation}
where $ l-1 \leq \alpha<l$ with $l\in\mathbb{N}^{\ast}$, and $\Gamma(\cdot)$ is the  Gamma function (see \cite{Abramowitz}
p. 255). As an example,  using (\ref{Eq_RiemannLiouvillle})
the fractional derivative of an $n^{th}$ ($n\in\mathbb{N}$) order polynomial
is given by
(see \cite{Podlubny2} p. 72): $\forall \, t \in \mathbb{R}_+^*$,
\begin{align}
{\mathrm{D}}_{t}^{\alpha}t^{n}&=\frac{\Gamma(n+1)}{\Gamma(n+1-\alpha
)}\, t^{n-\alpha}.\label{Eq_Derivative_poly}
\end{align}

We assume that the fractional derivative and the Laplace transform of $f$ both exist, then
the Laplace transform of the fractional derivative of $f$ is given by (see \cite{Kilbas} p. 284): $\forall \, s \in \mathbb{C}$,
\begin{align}
\mathcal{L}\left\{{\mathrm{D}}_{t}^{\alpha}f(t)\right\}(s)&=s^{\alpha}\hat{f}(s) -\sum_{i=0}^{l-1} s^i \left[{\mathrm{D}}_{t}^{\alpha-i-1}f(t)\right]_{t=0},\label{Eq_Laplace1}
\end{align}
where  $\hat{f}$ denotes the Laplace transform of $f$, and $s$ denotes the variable in the frequency domain.

Finally,  we recall some results on  the existence and the initial values of the fractional derivative in the following proposition.
\begin{prop}(see \cite{Podlubny2} pp. 75-77) \label{Prop_0}
If $f \in \mathcal{C}^{l-1}([0,T])$ and $f^{(l)} \in \mathcal{L}([0,T])$ where $T \in \mathbb{R}^*_+$, then
the Riemann-Liouville fractional derivative ${\mathrm{D}}_{t}^{\alpha}f(t)$ exists, where $ l-1 \leq \alpha<l$ with $l\in\mathbb{N}^{\ast}$.
Moreover, the condition $f^{(i)}(0)=0$, for $i=0,\ldots,l-1$, is equivalent to the following condition:
$\left[{\mathrm{D}}_{t}^{\alpha}f(t)\right]_{t=0}=0.$
\end{prop}


\subsection{Modulating functions}

Let $l \in \mathbb{N}^*$, $T \in \mathbb{R}^*_+$, and $g$ be a function satisfying  the following properties:
\begin{description}
  \item[$(P_1):$]
$g \in \mathcal{C}^l([0,T])$;

  \item[$(P_2):$]
$ g^{(j)}(0)= g^{(j)}(T)=0, \  \forall \,j=0,1,\dots, l-1$,
\end{description}
then $g$ is called  \emph{modulating function of order $l$}  \cite{Preising}.
Hence, according to Proposition \ref{Prop_0}, the modulating function $g$ has also the following properties:
\begin{description}
  \item[$(P_3):$]
$\forall \, 0 \leq \beta < l$, \ ${\mathrm{D}}_{t}^{\beta} g(t)$ exists;
  \item[$(P_4):$]
$\forall \, 0 \leq \beta < l$, \ $\left[{\mathrm{D}}_{t}^{\beta} g(t)\right]_{t=0}=0$.
\end{description}

\subsection{Fractional integration by parts}
Another useful  property of the modulating functions is given in the following theorem.
\begin{thh} \label{lemma1}
Let $y$ be a function such that the $\alpha^{th}$ order fractional derivative exits and  $g$ be a modulating function of order $l$ with $l-1 \leq \alpha<l$ with $l\in\mathbb{N}^{\ast}$. Then, we have:
\begin{equation}
 \int_0^T g(T-t)\, {\mathrm{D}}^{\alpha}_t y(t)\,dt=  \int_0^T  {\mathrm{D}}^{\alpha}_t  g(t)\, y(T-t)\,dt,
\end{equation}
where $T \in \mathbb{R}^*_+$.
\end{thh}

\begin{proof}
By applying the convolution theorem of the Laplace transform (see \cite{Abramowitz}, p. 1020), we get:
\begin{equation}\label{Eq_transf1}
\mathcal{L}\left\{\int_0^T g(T-t)\,{\mathrm{D}}^{\alpha}_t y(t)\,dt\right\}(s)=  \hat{g}(s) \mathcal{L}\left\{ {\mathrm{D}}^{\alpha}_t  y(t)\right\}(s).
\end{equation}
Then,  using (\ref{Eq_Laplace1}) we obtain:
\begin{equation}\label{Eq_transf2}
\begin{split}
&\hat{g}(s) \mathcal{L}\left\{ {\mathrm{D}}^{\alpha}_t  y(t)\right\}(s)=\\ &\qquad
\hat{g}(s) s^{\alpha}\hat{y}(s) -\sum_{i=0}^{l-1} s^i \hat{g}(s) \left[{\mathrm{D}}_{t}^{\alpha-i-1} y(t)\right]_{t=0}.
\end{split}
\end{equation}
Moreover, using (\ref{Eq_Laplace1}) and $(P_4)$  we get:
\begin{align}
 \mathcal{L}\left\{ {\mathrm{D}}^{\alpha}_t  g(t)\right\}(s)=
s^{\alpha}\, \hat{g}(s),\label{Eq_transf31}\\
 \mathcal{L}\left\{g^{(i)}(t)\right\}(s)=
s^{i}\, \hat{g}(s), \label{Eq_transf32}
\end{align}
for $i=0,\ldots,l-1$.
Consequently, by applying (\ref{Eq_transf31}), (\ref{Eq_transf32})  and the inverse of the Laplace transform to (\ref{Eq_transf2}), we obtain:
\begin{equation}\label{Eq_transf4}
\begin{split}
&\mathcal{L}^{-1}\left\{\hat{g}(s) \mathcal{L}\left\{ {\mathrm{D}}^{\alpha}_t  y(t)\right\}(s)\right\}(T)=\\ &\qquad  \qquad
\mathcal{L}^{-1}\left\{\mathcal{L}\left\{ {\mathrm{D}}^{\alpha}_t  g(t)\right\}(s)\ \hat{y}(s)\right\}(T) \\ &\qquad \qquad \qquad -\sum_{i=0}^{l-1} g^{(i)}(T) \left[{\mathrm{D}}_{t}^{\alpha-i-1}y(t)\right]_{t=0}.
\end{split}
\end{equation}
Using $(P_2)$,  the initial conditions $\left[{\mathrm{D}}_{t}^{\alpha-i-1}y(t)\right]_{t=0}$, for $i=0,1,\dots, l-1$,
can be eliminated.

Finally, this proof can be completed by applying the convolution theorem of the Laplace transform  to (\ref{Eq_transf4}).
\end{proof}

According to the previous theorem, we can see that by working in the frequency domain: on the one hand, we can obtain an integral formula which can be considered as the generalization  of the classical integration by parts formula.
Another fractional integration by parts formula has also been  given in \cite{Podlubny}, however the Caputo fractional derivative was involved in the formula; on the other hand, the initial conditions of the fractional derivatives of $y$ can be eliminated using a modulating function. In fact, the idea of obtaining this theorem is inspired by the recent algebraic parametric estimation technique \cite{mfhsr,garnier,Mboup2009b,Liu2011d,Liu2008,Wilfrid1,Wilfrid2,trapero}, which eliminates the unknown initial conditions by applying algebraic manipulations in the frequency domain.

\section{Identification of fractional order  systems} \label{section3}

\subsection{Fractional order linear systems}

In this section, we consider a class of   fractional order  linear systems which are defined by the following fractional differential equation:
$\forall\, t \in I=[0,T] \subset \mathbb{R}^*_+$,
\begin{equation}\label{Eq_problem1}
 \sum_{i=0}^{L} a_i \, {\mathrm{D}}^{\alpha_i}_t y(t) = \sum_{j=0}^{M} b_j \, {\mathrm{D}}^{\beta_j}_t u(t),
\end{equation}
where $y$ is the output, $u$ is the input,    $a_i, b_j \in  \mathbb{R}^*$ are unknown parameters to be identified, and $\alpha_i, \beta_j \in \mathbb{R}_+$
 are assumed $0 \leq \alpha_0<\alpha_1<\cdots<\alpha_L$, $0 \leq \beta_0<\beta_1<\cdots<\beta_M$ with $L,M \in \mathbb{N}$.

Let  $y^{\varpi}$ be a noisy observation of $y$ on the
interval $I$:
\begin{equation} \label{Eq_signal}
\forall\, t\in I, \ y^{\varpi}(t)=y(t)+\varpi(t),
\end{equation}
where  $\varpi$ is
an integrable noise\footnote{More generally, the noise is a
stochastic process, which is  integrable
in the sense of convergence in mean square \cite{Liu2011d}.}. We are going to estimate the unknown parameters in (\ref{Eq_problem1}) using the input $u$ and the observation $y^{\varpi}$.

One of the standard methods for the identification of systems with integer order derivatives is to use the least-squares method \cite{Ljung}. This method was generalized to fractional order systems in \cite{OuLeMa:96, DjVoCh:12}. In this method, we need to estimate the fractional derivatives of $y$ and $u$ using a fractional order differentiator \cite{Chen, Liu2012a,Liu2012b}.
However, a fractional order differentiator often contains a truncated term error. An alternative method consists in solving the problem in the frequency domain by applying the Laplace transform to (\ref{Eq_problem1}). However, according to (\ref{Eq_Laplace1}), this application can produce  unknown initial conditions.
In the next  subsection, we are going to apply  modulating functions method to eliminate these unknown initial conditions.


\subsection{Application of modulating functions method}

We denote $W=L+M+1$, $\alpha=\max\left(\alpha_L,\beta_M\right)$, and $l=\lceil \alpha \rceil$, where $\lceil \alpha \rceil$ denotes the smallest integer greater than or equal to $\alpha$.
Then, we take a set of modulating  functions $\left\{g_{n}\right\}_{n=1}^{N}$ with $W \leq N \in \mathbb{N}$.
Using Theorem \ref{lemma1}, we can give the following proposition.

\begin{prop}\label{Prop_1}
Let $\left\{g_{n}\right\}_{n=1}^{N}$ be a set of modulating functions of order $l$. If we assume that $b_0=1$  in  the fractional order  linear system defined by (\ref{Eq_problem1}), then the unknown parameters in this fractional order  linear system
can be estimated by solving the following linear system:
\begin{align}\label{Eq_system}
\left(\begin{array}{c}
   {U}_N \
   {Y}_N^{\varpi} \\
  \end{array}
\right)
\left(\begin{array}{c}
    {\tilde{B}} \\
    {\tilde{A}} \\
  \end{array}
\right) =
{I}_N,
\end{align}
where ${\tilde{B}}=\left({\tilde{b}}_1,\cdots,{\tilde{b}}_M\right)^T$, ${\tilde{A}}=\left({\tilde{a}}_0,\cdots,{\tilde{a}}_L\right)^T$ are estimators of the unknown parameters, and
\begin{align}
{U}_N(n,j)&= -\int_0^T  {\mathrm{D}}^{\beta_j}_t g_{n}(t)\, u(T-t) \,dt, \label{Eq_intg1}\\
{Y}_N^{\varpi}(n,i+1)&= \int_0^T {\mathrm{D}}^{\alpha_i}_t g_{n}(t)\, y^{\varpi}(T-t) \,dt,\label{Eq_intg2} \\
{I}_N(n)&=\int_0^T {\mathrm{D}}^{\beta_0}_t g_{n}(t)\, u(T-t) \,dt, \label{Eq_intg3}
\end{align}
for $n=1,\ldots,N$, $j=1,\ldots, M$ and $i=0,\ldots, L$.
\end{prop}

\begin{proof}
By multiplying the modulating functions $g_{n}$ to the equation (\ref{Eq_problem1}) and by integrating between $0$ and $T$, we get:
\begin{equation}\label{Eq_problem2}
\begin{split}
& \int_0^T \sum_{i=0}^{L} a_i \, g_{n}(T-t)\,{\mathrm{D}}^{\alpha_i}_t y(t) \,dt =\\
 &\qquad  \qquad  \int_0^T \sum_{j=0}^{M} b_j \, g_{n}(T-t)\,{\mathrm{D}}^{\beta_j}_t u(t) \,dt,
\end{split}
\end{equation}
for $n=1,\ldots,N$. Then,  using Theorem \ref{lemma1}, we get
\begin{equation}\label{Eq_problem3}
\sum_{j=1}^{M} b_j {U}_N(n,j) + \sum_{i=0}^{L} a_i {Y}_N(n,i+1) = \int_0^T {\mathrm{D}}^{\beta_0}_t g_{n}(t) u(T-t) dt,
\end{equation}
where ${U}_N(n,j)=-\int_0^T  {\mathrm{D}}^{\beta_j}_t g_{n}(t)\, u(T-t) \,dt$, and\\
${Y}_N(n,i+1)= \int_0^T {\mathrm{D}}^{\alpha_i}_t g_{n}(t)\, y(T-t) \,dt$.
Finally, this proof can be completed by substituting $y$ by $y^{\varpi}$ in (\ref{Eq_problem3}).
\end{proof}

Consequently,  thanks to Theorem \ref{lemma1}, instead of estimating (resp. calculating) the fractional derivatives of $y$ (resp. $u$),  we calculate the fractional derivatives of the modulating functions. On the one hand, comparing to $y$,  the  modulating functions are known and without noise.
On the other hand, if the fractional derivatives of $u$ cannot be analytically calculated or  are difficult to calculate, we can solve the problem by calculating the ones of the  modulating functions.

Finally, let us mention that  since the proposed estimators are given in causal case, if we take  the  value of $T$ to be equal to the time where we estimate the parameters, then these estimators can be used for on-line identification applications.

\section{Error analysis} \label{section4}

In this section, we are going to study the noise effect in the integrals obtained in Proposition \ref{Prop_1}. For this purpose, we study the noise error contributions due to a high frequency sinusoidal noise and the ones due to a class of stochastic processes in continuous case and in discrete case, respectively.

\subsection{Error analysis in continuous case}
There are many applications where the output
signal is corrupted by a sinusoidal noise of
higher frequency \cite{trapero}.
Hence, we assume that the noise $\varpi$ is a high frequency sinusoidal noise in this subsection.

By writing $y^{\varpi}=y+\varpi$, the integral  ${Y}_N^{\varpi}(n,i+1)$ given in (\ref{Eq_intg2}), for $n=1,\ldots,N$ and $i=0,\ldots, L$, can be
divided into:
\begin{equation} \label{Eq_cont}
{Y}_{N}^{\varpi}(n,i+1)= {Y}_{N}(n,i+1) + e_{N}^{\varpi}(n,i+1),
\end{equation}
where ${Y}_{N}(n,i+1)$ is given in (\ref{Eq_problem3}), and  the
associated noise error contribution is given by:
\begin{equation} \label{Eq_error}
e_{N}^{\varpi}(n,i+1)= \int_0^T {\mathrm{D}}^{\alpha_i}_t g_{n}(t)\, \varpi(T-t) \,dt.
\end{equation}
Consequently, the estimation errors  for the estimators given in Proposition \ref{Prop_1}
only come from these noise error
contributions. In the following proposition, we give error bounds for these noise error
contributions.

\begin{prop}
We assume that $\forall \, t \in I$, $\varpi(t)=c \sin(\omega t+ \phi)$  with $c, \omega \in \mathbb{R}^*_+$ and $\phi \in [0,2\pi[$.
Moreover, we assume that ${\mathrm{D}}^{\alpha_i+1}_t g_{n}(t)$ exists and is continuous on $[0,T]$, for $n=1,\ldots,N$ and $i=0,\ldots, L$.
Then, we have:
\begin{equation}
\left| e_{N}^{\varpi}(n,i+1)\right|  \leq  \frac{c }{\omega} T C_{\alpha_i+1}   + \frac{c }{\omega} \left| \left[{\mathrm{D}}_{t}^{\alpha_i} g_{n}(t)\right]_{t=T} \right|,
\end{equation}
where $C_{\alpha_i+1}=\displaystyle\sup_{t \in [0,T]}\left| {\mathrm{D}}^{\alpha_i+1}_t g_{n}(t) \right|$, and $e_{N}^{\varpi}(n,i+1)$ is given by  (\ref{Eq_error}).
\end{prop}

\begin{proof} Since   $\frac{d}{d t}\left\{{\mathrm{D}}^{\alpha_i}_t g_{n}(t)\right\}={\mathrm{D}}^{\alpha_i+1}_t g_{n}(t)$ exists, then
by applying integration by parts and $(P_4)$, we get:
\begin{equation} \label{Eq_ineg}
\begin{split}
e_{N}^{\varpi}(n,i+1) = &  \int_0^T {\mathrm{D}}^{\alpha_i}_t g_{n}(t)\, c\sin(\omega(T-t)+\phi) \,dt\\
= & -\frac{c}{\omega} \int_0^T {\mathrm{D}}^{\alpha_i+1}_t g_{n}(t)\, \cos(\omega(T-t)+\phi) \,dt\\
& \quad +  \frac{c}{\omega} \cos(\phi) \left[{\mathrm{D}}_{t}^{\alpha_i} g_{n}(t)\right]_{t=T}.
\end{split}
\end{equation}
If ${\mathrm{D}}^{\alpha_i+1}_t g_{n}(t) \in \mathcal{C}([0,T])$, then this proof can be completed using (\ref{Eq_ineg}).
\end{proof}

According to the previous proposition, if the frequency of the sinusoidal noise is high, then the associated noise error contributions can be negligible. Consequently, the estimators given in Proposition \ref{Prop_1} can cope with this kind of noises.

\subsection{Error analysis in discrete case}

From now on, we assume that the noisy observation $y^{\varpi}$ defined in (\ref{Eq_signal}) is  given in a discrete
case.
Let $y^{\varpi}(t_j)= y(t_j) + \varpi(t_j)$ be a noisy
discrete observation of $y$ given with an equidistant sampling period  $T_s$, where
$T_s = \frac{T}{m}$, $m \in \mathbb{N}^*$, and $t_j=j T_s$,
for $j=0,\cdots,m$.

Since $y^{\varpi}$ is a discrete
measurement, we  apply a numerical integration method
to approximate the integrals  in (\ref{Eq_system}). Let
 $w_0 \geq 0$, $w_m
\geq 0$ and $w_j>0$ for $j=1,\dots, m-1$    be the weights for a given
numerical integration method, where weight $a_0$ (resp. $a_m$) is set to
zero when there is an infinite value at $t_0=0$  (resp. $t_m=T$).
Then, the integral ${Y}_N^{\varpi}(n,i+1)$ given in (\ref{Eq_intg2}), for $n=1,\ldots,N$ and $i=0,\ldots, L$, can be
approximated by:
\begin{equation}
{Y}_{N,m}^{\varpi}(n,i+1):= T_s \sum_{j=0}^m
{w_j}\, g_{n}^{(\alpha_i)}(t_j)\, y^{\varpi}(t_{m-j}),
\end{equation}
where $g_{n}^{(\alpha_i)}(t_j):=\left[{\mathrm{D}}^{\alpha_i}_t g_{n}(t)\right]_{t=t_j}$.
The integrals  given in (\ref{Eq_intg1}) and (\ref{Eq_intg3})  can be approximated in a similar way.

By writing  $y^{\varpi}(t_j)=
y(t_j) + \varpi(t_j)$, we get:
\begin{equation}
{Y}_{N,m}^{\varpi}(n,i+1)= {Y}_{N,m}(n,i+1) + e_{N,m}^{\varpi}(n,i+1),
\end{equation}
where
\begin{align}
{Y}_{N,m}(n,i+1)&= T_s \displaystyle\sum_{j=0}^m
{w_j}\, g_{n}^{(\alpha_i)}(t_j)\, y(t_{m-j}), \\
e_{N,m}^{\varpi}(n,i+1) &= T_s \displaystyle\sum_{j=0}^m
{w_j}\, g_{n}^{(\alpha_i)}(t_j)\, \varpi(t_{m-j}). \label{Eq_discrete}
\end{align}
Thus the integral ${Y}_N^{\varpi}(n,i+1)$ is corrupted
by two sources of errors:
\begin{itemize}
  \item  the  numerical error which comes from a  numerical integration method,
  \item the noise error
contribution $e_{N,m}^{\varpi}(n,i+1)$.
\end{itemize}
Consequently, the estimation errors  for the estimators given in Proposition \ref{Prop_1}
come from both the numerical errors and the noise error
contributions in  the discrete noisy case.

It is well known that if the value of $T$ is set, then when $T_s$ tends to $0$, $i.e.$   $m \rightarrow +\infty$, the numerical errors tend to $0$.   In the next subsection, we are going to study the effect of the sampling period
on the noise error contributions.

\subsection{Influence of sampling period  on  noise error contributions}

In this subsection,  we consider a family of noises which are stochastic processes satisfying  the following
conditions:
\begin{description}
  \item[$(C_1):$] for any $t,s \in I$, $t\neq s$, $\varpi(t)$ and $\varpi(s)$ are
  independent;
  \item[$(C_2):$] the mean value function of $\varpi(\cdot)$ denoted by $\mathrm{E}[\cdot]$ belongs to $\mathcal{L}(I)$;
  \item[$(C_3):$] the variance function of $\varpi(\cdot)$ denoted by $\mathrm{Var}[\cdot]$ is bounded on $I$.
\end{description}
Note that the white Gaussian noise and the Poisson noise satisfy these
conditions. Then, we can give the following proposition.

\begin{prop} \label{Prop_convergence_Ts}
Let $\varpi(\cdot)$ be a stochastic process satisfying conditions $(C_1)-(C_3)$, and
$\varpi(t_j)$, for $j=0,\cdots, m$, be a sequence of $\{\varpi(\cdot)\}$ with
an equidistant sampling period $T_s$. If  ${\mathrm{D}}^{\alpha_i}_t g_{n}(t) \in \mathcal{L}^2(I)$, then we have the following convergence in mean square of the noise error contribution in the integral ${Y}_{N}^{\varpi}(n,i+1)$ given in (\ref{Eq_intg2}), for $n=1,\ldots,N$ and $i=0,\ldots, L$:
\begin{equation} \label{Eq_convergence1}
e_{N,m}^{\varpi}(n,i+1) \xLongrightarrow[T_s \rightarrow 0]{\mathcal{L}^2(I)}   \int_{0}^{T}  {\mathrm{D}}^{\alpha_i}_t g_{n}(t)\, \mathrm{E}\left[\varpi(T-t)\right]\, d t,
\end{equation}
where $e_{N,m}^{\varpi}(n,i+1)$ is given in (\ref{Eq_discrete}).
Moreover,
 if $\forall\, t \in I$, $\mathrm{E}\left[\varpi(t)\right]=0$, then we have:
\begin{equation} \label{Eq_convergence2}
e_{N,m}^{\varpi}(n,i+1) \xLongrightarrow[T_s \rightarrow 0]{\mathcal{L}^2(I)} 0.
\end{equation}
\end{prop}

The proof of the previous proposition can be obtained
in a similar way to the one given in  \cite{Liu2011b}. Moreover, a similar result was
studied  using the non-standard analysis in \cite{ans}.

Consequently, according to the previous proposition, the noise error contributions can be increasing with respect to the sampling period.

Finally, let us mention  that solving the linear system given in Proposition \ref{Prop_1} in noisy case is related to the matrix perturbation theory. 
 The accuracy and the stability of the proposed estimators not only depend on the noise error contributions, but also depend on the condition number of the associated matrix. This condition number depends both  on the input and the output of the  fractional order system and on the used modulating functions. In general, we should choose
the modulating functions that can give a small condition number. This study is out of the scope of this paper.

\section{Simulation results} \label{section5}

In order to illustrate  the accuracy and robustness with respect to corrupting noises of the
proposed estimators, we present some numerical results in
this section.

Let us consider a fractional order system defined by the following fractional differential equation: $\forall \, t \in [0,8]$,
\begin{equation}
a_0  {\mathrm{D}}^{\alpha_0}_t y(t) + a_1  {\mathrm{D}}^{\alpha_1}_t y(t)  + a_2 {\mathrm{D}}^{\alpha_2}_t y(t)  =  u(t),
\end{equation}
where $a_0=3$, $a_1=2$, $a_2=1$,  $\alpha_0=0$, $\alpha_1=0.5$ and $\alpha_2=1.5$. We assume that the output is  $y(t)=\sin(3t)+1$. Hence, the initial conditions of $y$ are not equal to $0$. Moreover, the expression of the input can be obtained  using (\ref{Eq_Derivative_poly}) and the following formula  (see
\cite{Miller} p. 83):
\begin{equation}\label{}
\begin{split}
&{\mathrm{D}}_{t}^{\alpha_i} \sin(3t) =\\&    \frac{3 t^{1-\alpha_i}}{\Gamma(2-\alpha_i)}\,{}_{1}{\mathrm{F}}_2\left(1;\frac{1}{2}(2-\alpha_i),\frac{1}{2}(3-\alpha_i);-\frac{1}{4}3^2 t^2\right),
\end{split}
\end{equation}
where ${}_{p}{\mathrm{F}}_q\left(c_1,\ldots,c_p; d_1,\ldots,d_q; \cdot\right)$
is the generalized hypergeometric function given in
\cite{Miller} p. 303.

In our identification procedure, we use the following modulating functions, the fractional derivative of which are simple to calculate:
\begin{equation}
g_{n}(t)=(T-t)^{6+n} t^{6+N+1-n},
\end{equation}
where  $n=1,2,\cdots,N$ with $N=13$.  The value of $T$ is taken to be equal to the time where we estimate the parameters. Let us recall that this kind of functions has been obtained when the algebraic parametric estimation technique was applied to the
parameter estimation for signals described by
differential equations \cite{Mboup2009b}.

In the two following examples, we estimate the parameters $a_0$, $a_1$ and $a_2$  using the noisy observation of $y$ where the noise is a high frequency sinusoidal noise and a gaussian noise, respectively. Moreover, we apply the trapezoidal rule to numerically approximate the integrals in our estimators.

{\textbf{Example 1.}} In this example, we assume that $y^{\varpi}(t_j)=y(t_j)+0.5\sin(10^3 t_j)$ with  $T_s=0.01$. We can see this discrete noisy signal in Figure \ref{Fig_signal1}. In our identification procedure, we take $T=t_i$ for $t_i \in [1.8,8]$.
The obtained estimations and the associated relative estimation errors are shown in Figure \ref{Fig_estimation1} and Figure \ref{Fig_error1}.
Hence, we can see that the proposed estimators  can cope with a high frequency sinusoidal noise.

{\textbf{Example 2.}}
In this example, we assume that
 $y^{\varpi}(t_j)=y(t_j)+\sigma\varpi
(t_j)$, where $T_s=0.01$,  $\varpi(t_j)$ is simulated from a zero-mean white Gaussian $iid$ sequence, and $\sigma \in \mathbb{R}_+^*$ is
adjusted in such a way that the signal-to-noise ratio  is equal to
$SNR=22\text{dB}$. This noisy observation is shown in Figure \ref{Fig_signal2}. The  obtained estimations and the associated relative estimation errors are given  in Figure \ref{Fig_estimation2} and Figure \ref{Fig_error2}.
We can see that the proposed estimators are robust against a gaussian noise.

\begin{figure}[h!]
\centering
\framebox{\parbox{2.5in}{
 {\includegraphics[scale=0.44]{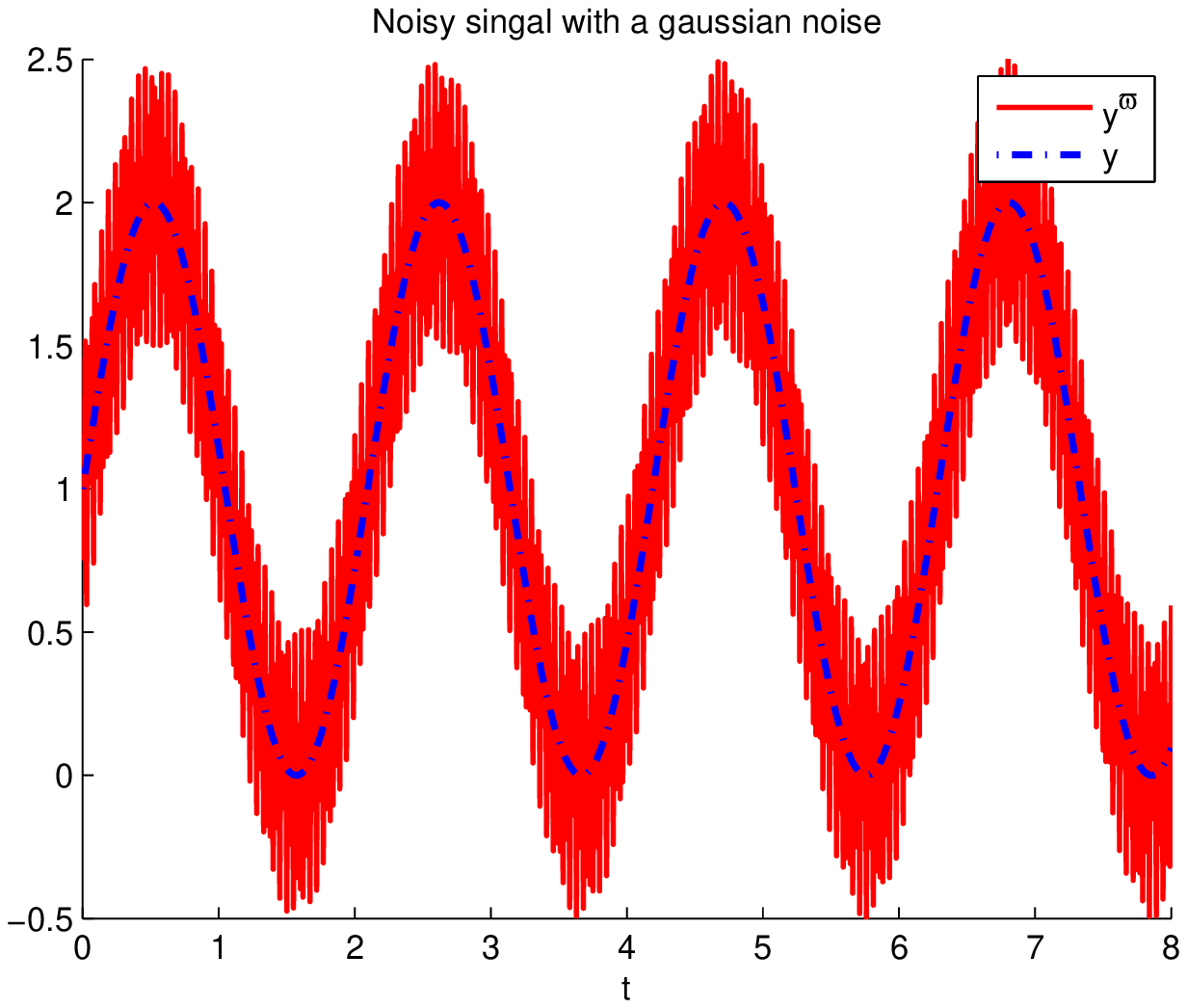}}}}
\caption{The noise-free output and its noisy observation.} \label{Fig_signal1}
\end{figure}

\begin{figure}[h!]
\centering
\framebox{\parbox{2.5in}{
 {\includegraphics[scale=0.44]{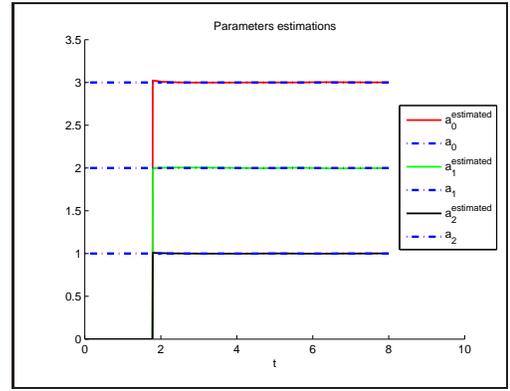}}}}
\caption{The exact parameters and their estimations with a sinusoidal noise.} \label{Fig_estimation1}
\end{figure}

\begin{figure}[h!]
\centering
\framebox{\parbox{2.5in}{
 {\includegraphics[scale=0.44]{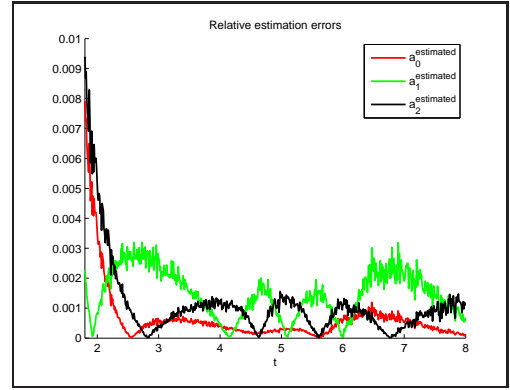}}}}
\caption{Relative estimation errors in the  sinusoidal noise case.} \label{Fig_error1}
\end{figure}

\begin{figure}[h!]
\centering
\framebox{\parbox{2.5in}{
 {\includegraphics[scale=0.44]{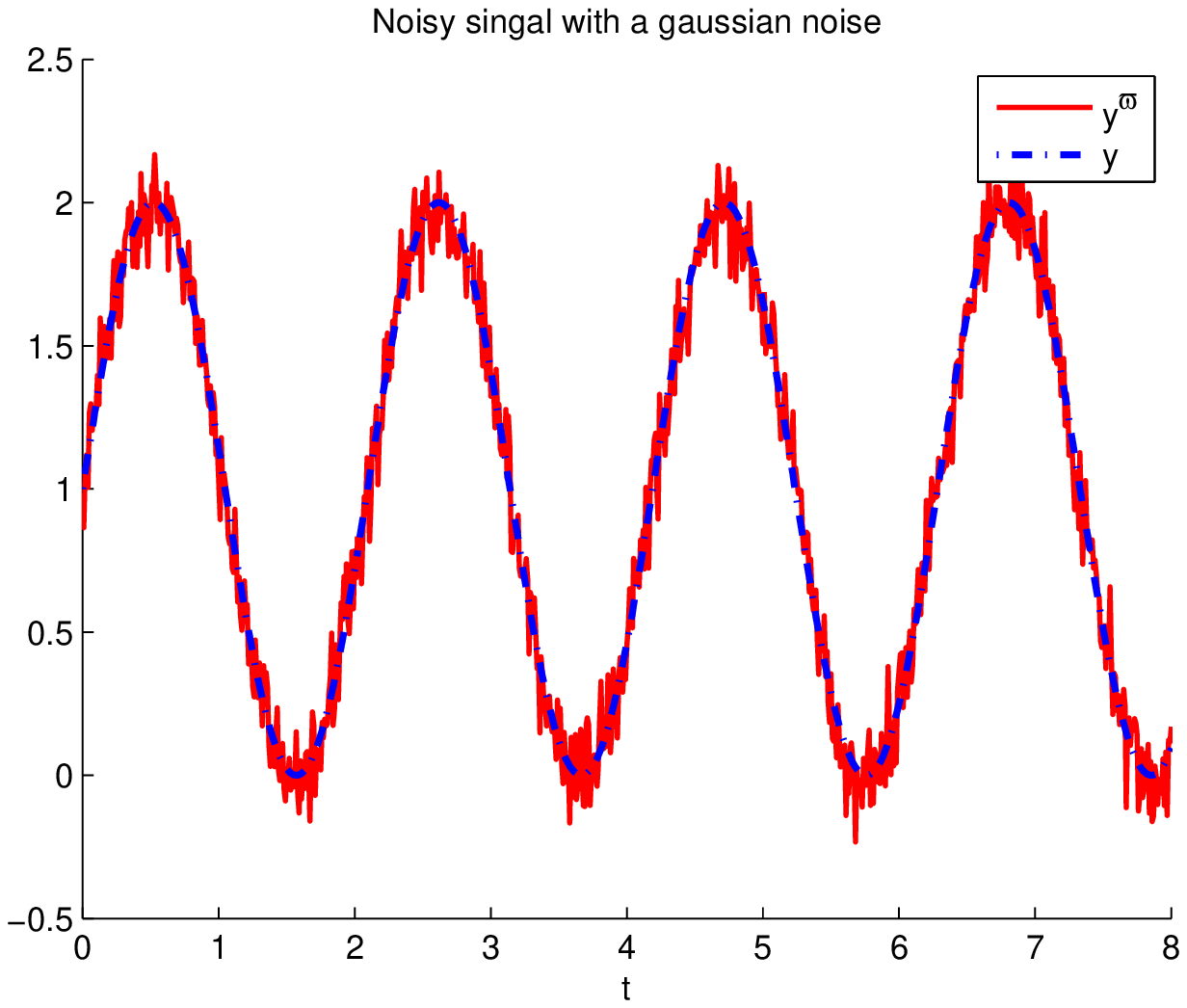}}}}
\caption{The noise-free output and its noisy observation.} \label{Fig_signal2}
\end{figure}

\section{Conclusions} \label{section6}

In this paper, the modulating functions method has been generalized to the on-line identification problem of fractional order systems. Using this method, the unknown parameters have been estimated by solving a linear system of algebraic equations involving the input and the noisy output. Thanks to the properties of modulating functions, we do not need to estimate the fractional derivatives of the output, to calculate the ones of the input. We do not need to know the initial conditions either. Moreover, the integral given in the estimators can reduce the noise effects due to a high frequency sinusoidal noise or a class of stochastic processes. The efficiency and the robustness against corrupting  noises have been confirmed by numerical examples. In order to improve the robustness against noises, some methods such as the instrumental variable method will be applied \cite{Garnier}.   It was mentioned that the proposed estimators also depend on the choice of modulating functions. This problem will be studied in the future work. Moreover, the estimation of the fractional derivative orders in a time-delayed fractional order system will be considered.

\begin{figure}[h!]
\centering
\framebox{\parbox{2.5in}{
 {\includegraphics[scale=0.44]{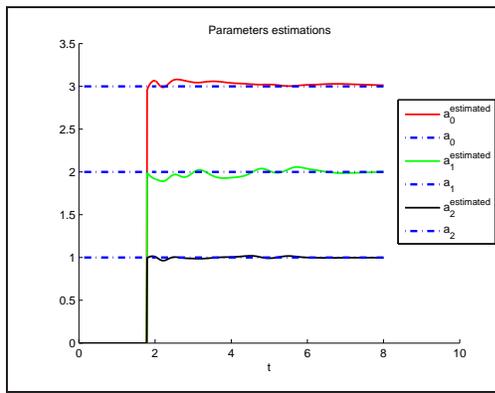}}}}
\caption{The exact parameters and their estimations with a gaussian noise.} \label{Fig_estimation2}
\end{figure}

\begin{figure}[h!]
\centering
\framebox{\parbox{2.5in}{
 {\includegraphics[scale=0.44]{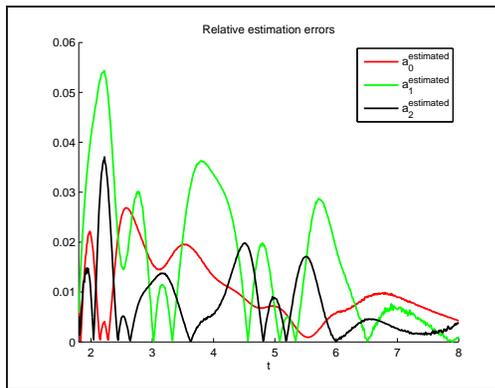}}}}
\caption{Relative estimation errors in the gaussian noise case.} \label{Fig_error2}
\end{figure}

\end{document}